\documentclass[12pt,a4paper]{article}
\usepackage{amsmath}
\usepackage{amsthm}
\usepackage{amsfonts}
\usepackage{amssymb}
\usepackage{stmaryrd}
\usepackage{latexsym}
\usepackage{eucal}

\theoremstyle{theorem}
\newtheorem{theorem}{Theorem}
\newtheorem{lemma}[theorem]{Lemma}

\newtheorem{note}[theorem]{Note}

\newtheorem{definition}[theorem]{Definition}

\theoremstyle{definition}
\newtheorem{remark}[theorem]{Remark}
\newtheorem{example}[theorem]{Example}

\renewenvironment{proof}[1][\proofname]{\noindent \textbf{Proof. }}
{%
  \qed\endtrivlist
}

\begin{document}

  \author{G.~Horosh, N.~Malyutina, A.~Scerbacova, V.~Shcherbacov}

\title{Units in generalized derivatives of quasigroups}

\maketitle

\begin{abstract}
We proceed  the research of generalized quasigroup derivatives started in early papers of the last co-author
(\cite[p. 212]{2017_Scerb}, \cite{Krapez_Scerb}). For any quasigroup there exist 648 generalized derivatives. Here we study the problem  about   existence of units (left, right, middle) in quasigroups that are a generalized derivative of a quasigroup. There exist 1944 various cases. For every  case we find a proof or a counterexample, often  using Prover or Mace \cite{MAC_CUNE_PROV, MAC_CUNE_MACE}.

\medskip


\medskip

\end{abstract}

\section{Introduction}

In this paper  we further  extend Belousov's concept  of derivatives \cite{VD, Krapez_Scerb, Krapez_19}. From the early works of V.D. Belousov it follows that this concept  is closely related to the  Belousov's  Problem \# 18. This Problem  has the following formulation \lq\lq How to recognize identities which
force quasigroups satisfying them to be loops? \rq\rq \, \cite[Problem \#18, page 217]{VD}. Notice, Belousov writes that this Problem was proposed by I.E. Burmistrovich.

In \cite{Krapez_Scerb}  Belousov's  problem is formulated in more general form \lq\lq How to recognize identities which force quasigroups satisfying them to have left, right, middle unit?\rq\rq.

This paper prolongs and extends  researches started in \cite[p. 212]{2017_Scerb}, \cite{Krapez_Scerb}.

Basic concepts can be found in \cite{VD, HOP, 2017_Scerb}. We have used very actively computer tools Prover 9 and Mace 4 \cite{MAC_CUNE_PROV, MAC_CUNE_MACE}, which were developed by Professor W. McCune that left our world early.

\subsection{Quasigroup}

Garrett Birkhoff  \cite{BIRKHOFF_1948, BIRKHOFF} has  defined an equational quasigroup as an algebra with three binary operations $(Q, \cdot, \slash, \backslash)$ that satisfies  the following six identities:
\begin{equation}
x\cdot(x \backslash y) = y, \label{(1)}
\end{equation}
\begin{equation}
(y / x)\cdot x = y, \label{(2)}
\end{equation}
\begin{equation}
x\backslash (x \cdot y) = y, \label{(3)}
\end{equation}
\begin{equation}
(y \cdot x)/ x = y, \label{(4)}
\end{equation}
\begin{equation}
x/(y\backslash x) = y,  \label{T}
\end{equation}
\begin{equation}
(x/y)\backslash x = y. \label{R}
\end{equation}

\begin{remark}
In \cite{JDH_2007} the identities (\ref{(1)})--(\ref{(4)}) are called  (SL), (SR), (IL), (IR), respectively, since these identities guarantee that the left (L)  and right (R) translations of an algebra $(Q, \cdot, \slash, \backslash)$ relative to the operation \lq\lq $\cdot$" are surjective (S) or injective (I) mappings of the set $Q$.

Following this  logic we can denote  identity (\ref{T}) by  (SP) and identity (\ref{R}) by (IP) since these identities guarantee that middle translations (P) are respectively surjective and injective mappings relative to the operation "$\cdot$" \cite{SCERB_07}. Indeed, using Table \ref{Table_0}, we see that $L^{/}_x = P^{-1}_x$.
\end{remark}

The following lemma is well known:

\begin{lemma}\label{main_four_ident_lemma}
In algebra  $(Q, \cdot, \backslash, /)$ with  identities  (\ref{(1)})--(\ref{(4)}),  identities (\ref{T}) and (\ref{R}) are true   \cite{JDH_2007, SCERB_03, SCERB_07}.
\end{lemma}

Therefore  the following Evans  equational definition of a quasigroup is usually used  \cite{EVANS_49}.
\begin{definition}  \label{EQUT_QUAS_DEF} \cite{BIRKHOFF_ENG, BIRKHOFF, EVANS_49}.
 A groupoid $(Q, \cdot)$ is called a quasigroup if, on the set $Q$, there exist
operations \lq\lq $\backslash$" and \lq\lq $/$" such that in the algebra $(Q, \cdot, \backslash, /)$
identities  (\ref{(1)})--(\ref{(4)}) are fulfilled. \index{quasigroup!definition!equational}
\end{definition}

\subsection{Parastrophes}

Here   we mainly use \cite{2017_Scerb}.

\begin{definition} \label{def2} An $n$-ary groupoid $(Q, A)$ with $n$-ary operation $A$ such
that in the equality $A(x_1, x_2, \dots, x_n) = x_{n+1}$ the fact of knowing  any $n$   elements of the set   $\{x_1, x_2, \dots,$
$ x_n, x_{n+1}\}$  uniquely specifies the remaining one element, is called an $n$-ary quasigroup  \cite{2}. \index{quasigroup!$n$-ary}
\end{definition}

 If we put  $n=2$, then  we obtain one more definition of a binary quasigroup.

\begin{definition} \label{def2_2_2} \index{quasigroup!parastrophe} \index{parastrophy}
From Definition \ref{def2}  it  follows that  with a given binary quasigroup $(Q, A)$   it is possible to associate
$(3!-1)$ other so-called parastrophes of quasigroup $(Q, A)$:
\[
\begin{array}{ll}
1. & A(x_1, x_2) = x_3 \Longleftrightarrow \\
2. &  A^{(12)}(x_2, x_1) = x_3 \Longleftrightarrow \\
3. & {A}^{(13)}(x_3, x_2) = x_1 \Longleftrightarrow \\
4. & {A}^{(23)}(x_1, x_3) = x_2 \Longleftrightarrow \\
5. & {A}^{(123)}(x_2, x_3) = x_1 \Longleftrightarrow \\
6. & {A}^{(132)}(x_3, x_1) = x_2
\end{array}
\]
\cite[p. 230]{STEIN}, \cite[p. 18]{VD}.
\end{definition}
\begin{note} \label{Note_1_6}
Notice, Cases 5 and 6 are \lq\lq $(12)$-parastrophes" of Cases 3 and 4, respectively.
\end{note}

\subsection{Translations}

The following table (using Table \ref{Table_0})  shows for each kind of translation the equivalent one in each of the (six) parastrophes of a
quasigroup  $(Q,\cdot)$. In fact, Table \ref{Table_0} is a rewritten form of results on three kinds of translations from \cite{BELAS}. See also
\cite{DUPLAK, SCERB_07}.

\begin{table}[h!]
\caption{Translations of quasigroup parastrophes.}
\label{Table_0}
\begin{center}
\begin{tabular}{|c||c| c| c| c| c| c|}
\hline
  & $\varepsilon$  & $(12)$ & $(13)$ & $(23)$ & $(123)$ & $(132)$ \\
\hline\hline
$R$  & $R$ & $L$ & $R^{-1}$ & P & $P^{-1}$ & $L^{-1}$ \\
\hline
$L$  & $L$ & $R$ & $P^{-1}$ & $L^{-1}$ & $R^{-1}$ & $P$ \\
\hline
$P$  & $P$ & $P^{-1}$ & $L^{-1}$ & $R$ & $L$ & $R^{-1}$ \\
\hline
$R^{-1}$ & $R^{-1}$ & $L^{-1}$ & $R$ & $P^{-1}$ &  $P$ & $L$ \\
\hline
$L^{-1}$  & $L^{-1}$ & $R^{-1}$ & $P$ & $L$ & $R$ & $P^{-1}$ \\
\hline
$P^{-1}$  & $P^{-1}$ & $P$ & $L$ & $R^{-1}$ & $L^{-1}$ & $R$ \\
\hline
\end{tabular}
\end{center}
\end{table}

From Table \ref{Table_0} it follows,  for example, that $R^{(132)} = L^{-1} = L^{(23)} = P^{(13)} =
(R^{-1})^{(12)} = (P^{-1})^{(123)}$.

\subsection{Unit elements}

Suppose we have a quasigroup $(Q, \cdot)$.

\begin{definition}
The fact that an element $f\in Q$ is a left  identity element (left unit)  for quasigroup $(Q, \cdot)$ means
that $f\cdot x = x $ for all $x\in Q$.

The fact that an element $e\in  Q$ is a right identity element (right unit) for quasigroup $(Q,\cdot)$ means
that $x\cdot e = x$  for all $x \in  Q$.

The fact that an element $s\in Q$  is a middle  identity element  (middle unit) for quasigroup $(Q, \cdot)$  means
that $s = x\cdot x$  for all $x \in  Q$.

\end{definition}

\subsection{Classical and almost classical  derivatives}

 We follow \cite{vdb0, vdb00, VD, 1a, HOP, 2017_Scerb}. The main idea belongs to V.D. Belousov. See also more earlier  articles of D.G.~Murdoch and A.K.~Suschkewitsch \cite{MURD_39, SUSHKEV_BOOK}. It is clear that identity of associativity is not true in any quasigroup. But we can replace it with  the following equality which is true in any  quasigroup $(Q, A)$:
\begin{equation}\label{Derived_Operation}
A(A(a, b), c) = A(a, A_a(b, c)),
\end{equation}
where $a, b, c \in Q$,   $A_a$ is some binary operation which depends on  the element $a$. The equality (\ref{Derived_Operation}) can be obtained from the following equation $A(A(a, b), c) = A(a, x)$. In this case the solution of the equation can be denoted as $x = A_a(b, c)$.
\begin{definition} \label{Right_Derived_Operation}
The operation $A_a$, which is defined from the equation $A(A(a, b), c) = A(a, x)$, is called a right derivative operation of the operation $A$ relative to element $a$, i.e., $x = A_a(b, c)$  \cite{VD}.
\end{definition}

Quasigroup derivatives are used in  the study of $G$-loops \cite{VD, 1a}.

\begin{definition}\label{Left_Derived_Operation}
The operation ${}_aA$, which is defined from the equation $A(b, A(c,a)) = A(y, a)$,
 is called a left derivative operation of the operation $A$ relative to element $a$, i.e., $y = {}_aA(b, c)$ \cite{VD}.
\end{definition}
\index{operation!left derivative}
\index{quasigroup!left derivative}

Here we give an isotopical   approach to the concept of quasigroup derivatives  \cite{1a, HOP}. For a quasigroup $(Q, \cdot)$ we can rewrite the equality (\ref{Derived_Operation}) in the following form:
\begin{equation} \label{Left_Nucleus_Isotope}
(a\cdot x) \cdot y = a\cdot (x\circ y),
\end{equation}
where $x, y$ are arbitrary elements of the set $Q$ and  element $a$ is a fixed element of $Q$. The equality (\ref{Left_Nucleus_Isotope}) defines a groupoid $(Q, \circ)$.  Moreover, $(Q, \circ) = (Q, \cdot) (L_a, \varepsilon, L_a)$, i.e., the groupoid $(Q, \circ)$ is an isotope of quasigroup $(Q, \cdot)$ with isotopy $(L_a, \varepsilon, L_a)$. Therefore, the groupoid $(Q, \circ)$ is a quasigroup and it is a \textit{right}  derivative of quasigroup $(Q, \cdot)$ relative to element $a$ as in  Definition  \ref{Right_Derived_Operation}.

Quasigroup  $(Q, \ast) = (Q, \cdot) (\varepsilon, R_a, R_a)$ is a \textit{left} derivative of quasigroup $(Q, \cdot)$ with respect to element $a$ as in  Definition \ref{Left_Derived_Operation} \cite{1a}.

\begin{definition}\cite{2017_Scerb}.
Quasigroup  $(Q, \star) = (Q, \cdot) (R_a, L^{-1}_a, \varepsilon)$ is  called  a \textit{middle} derivative of quasigroup $(Q, \cdot)$ with respect to element $a$.

Quasigroup  $(Q, \diamond) = (Q, \cdot) (R^{-1}_a, L_a, \varepsilon)$ is  called a \textit{middle inverse} derivative of quasigroup $(Q, \cdot)$ with respect to element $a$. \index{quasigroup!derivative!middle inverse}
\end{definition}

\begin{lemma}\label{Identity_El_Derivat}
1. Any right derivative  $(Q, \circ)$ of a quasigroup $(Q, \cdot)$ has  left identity element, i.e., $(Q, \circ)$ is a left loop  \cite{1a}.

2. Any left derivative $(Q, \ast)$ of a quasigroup $(Q, \cdot)$ has  right identity element, i.e., $(Q, \ast)$ is a right loop \cite{1a}.

3. Any middle  derivative $(Q, \star)$ of a quasigroup $(Q, \cdot)$  is a left loop \cite{2017_Scerb}.

4. Any middle inverse derivative $(Q, \diamond)$ of a quasigroup $(Q, \cdot)$  is a right loop \cite{2017_Scerb}.
\end{lemma}

\begin{definition}
 Isotopisms of the form $(L_a, \varepsilon, L_a)$, $(\varepsilon, R_a, R_a)$, $(R^{-1}_a, L_a, \varepsilon)$, $(R_a, L^{-1}_a, \varepsilon)$ can be called  as nuclear isotopisms \cite{2017_Scerb}[Definition 4.12].
\end{definition}
Left  (right, middle) nuclear isotopic image of quasigroup $(Q, \cdot)$ can be named as  right (left, middle)  derivative of a quasigroup $(Q,\cdot)$.
See Chapter \lq\lq A-nuclei of quasigroups\rq\rq\,  about autotopy nuclei in \cite{2017_Scerb}.


\subsection{Isostrophical approach}

Above we developed  isotopical approach to binary derivatives.

Now we change binary isotopical approach to binary isostrophical approach \cite{2017_Scerb}.
Notice, other modification of concept of derivative is given in \cite{Krapez_19}.

By the letter $T$ we denote the set of all quasigroup translations of a fixed quasigroup $(Q, \cdot)$ and their inverses relatively one fixed element, say, relatively  an element $a$.

\begin{definition}
Quasigroup  $(Q, \star) = (Q, \cdot) (\alpha, \beta, \gamma)$, where $(Q, \star)$ is isostrophic image of  quasigroup $(Q, \cdot)$, i.e.,  $\cdot \in \{A, A^{(12)},  {A}^{(13)},  {A}^{(23)}, {A}^{(123)}, {A}^{(132)}\}$,   $\alpha, \beta, \gamma \in T$, and in every case one of the translations $\alpha, \beta, \gamma$ is an identity permutation, is  called  an \textit{isostrophic} (generalized) derivative of quasigroup $(Q, \cdot)$ with respect to element $a$.
\end{definition}

It is clear that any binary quasigroup has $648$ isostrophic derivatives (see Table \ref{TABLE_two_thre}).

\begin{theorem}
\begin{enumerate}
  \item Quasigroup  $(Q, \star) = (Q, \cdot) (L_a, L_a, \varepsilon)$ has left unit element;
  \item   quasigroup $(Q, \circ) = (Q, \ast) (L_a, L_a, \varepsilon)$, where $x\ast y = y\cdot x$ for all $x, y \in Q$,  has right unit element;
  \item  quasigroup $(Q, \bullet) = (Q, \backslash ) (L_a, L^{-1}_a, \varepsilon)$, where $x\backslash y = z$ if and only if $ x\cdot z = y$ for all suitable  $x, y, z \in Q$,  has left unit element.
\end{enumerate}
\end{theorem}
\begin{proof}
\begin{enumerate}
\item
We re-write equality $(Q, \star) = (Q, \cdot) (L_a, L_a, \varepsilon)$ in the following form:
\begin{equation} \label{equality_9}
x\star y = ax\cdot ay.
\end{equation}
We substitute $x=L^{-2}_a$ in equality (\ref{equality_9}). We have

\begin{equation} \label{equality_10}
L^{-2}_a\star y = y.
\end{equation}
\item

We re-write equality $(Q, \circ) = (Q, \ast) (L_a, L_a, \varepsilon)$ in the following form:
\begin{equation} \label{equality_19}
x\circ y = ax\ast ay = ay\cdot ax.
\end{equation}
We substitute  $y=L^{-2}_a$ in equality (\ref{equality_19}). We have

\begin{equation} \label{equality_11}
x \circ L^{-2}_a = x.
\end{equation}
\item

We re-write equality $(Q, \bullet) = (Q, \backslash ) (L_a, L^{-1}_a, \varepsilon)$ in the following form:
\begin{equation} \label{equality_29}
x \bullet y = ax \backslash ( a\backslash y).
\end{equation}
From equality (\ref{equality_29}) we have
\begin{equation} \label{equality_112}
ax\cdot  (x\bullet y) = a\backslash y, \qquad a(a(x(x\bullet y))) =y .
\end{equation}
If we substitute  $x=L^{-2}_a$ in equality (\ref{equality_112}),
then we have
\begin{equation} \label{equality_113}
L^{-2}_a\bullet y =y .
\end{equation}
\end{enumerate}
\end{proof}

\begin{example} \label{Example_Gay}
We demonstrate that in general, isostrophical  derivative of the form:
\begin{equation} \label{Equality_134}
x \diamond  y  = L_a x \backslash P^{-1}_a y \: (\textrm {Table 6, row 3})
\end{equation}
 has no left, right, middle unit.
Using table of translations (Table \ref{Table_0}),  we re-write  this isostrophical  derivative  in the following form $ x \diamond  y = (a\cdot x) \backslash (a/ y)$.

Using Table \ref{TABLE_two_thre},  we can say that this isostrophical  derivative $(Q, \diamond)$ has no left (right, middle) identity element for a  fixed element $a\in Q$. Indeed,
we take the following quasigroup $(Q, \cdot)$ (the cyclic group of order 3)

\begin{center}
\begin{tabular}{r|rrr}
$\cdot $ & 0 & 1 & 2\\
\hline
    0 & 0 & 1 & 2 \\
    1 & 1 & 2 & 0 \\
    2 & 2 & 0 & 1
\end{tabular}
\end{center}
 and $a:=0$. Its isostrophical image $(Q, \diamond)$ of the form (\ref{Equality_134})  has no left  and right unit
\begin{center}
\begin{tabular}{r|rrr}
$\diamond$ & 0 & 1 & 2\\
\hline
    0 & 0 & 2 & 1 \\
    1 & 2 & 1 & 0 \\
    2 & 1 & 0 & 2
\end{tabular}
\end{center}

We take the following quasigroup $(Q, \cdot)$
\begin{center}
\begin{tabular}{r|rrr}
$\cdot$ & 0 & 1 & 2\\
\hline
    0 & 1 & 2 & 0 \\
    1 & 0 & 1 & 2 \\
    2 & 2 & 0 & 1
\end{tabular}
\end{center}

 and $a:=0$. Its isostrophical image $(Q, \diamond)$ of the form (\ref{Equality_134})  has no middle   unit
\begin{center}
\begin{tabular}{r|rrr}
$\diamond$ & 0 & 1 & 2\\
\hline
    0 & 1 & 2 & 0 \\
    1 & 2 & 0 & 1 \\
    2 & 0 & 1 & 2
\end{tabular}
\end{center}

\end{example}

\subsection{Table}

In Table \ref{TABLE_two_thre} we research all 1944 cases when a generalized derivative of a quasigroup has a unit.
While  filling  Table \ref{TABLE_two_thre}, we actively used Prover and Mace \cite{MAC_CUNE_PROV,  MAC_CUNE_MACE}.

We denote main quasigroup operation $A$ as $xy$, $yx$ denotes the operation $A^{(12)}$,  symbol \lq\lq $y \slash x $\rq\rq denotes  operation ${A}^{(13)}$,  symbol \lq\lq $x \backslash y $\rq\rq denotes  operation ${A}^{(23)}$,  symbol  \lq\lq $x \slash y $\rq\rq denotes  operation ${A}^{(123)}$,  symbol \lq\lq $y \backslash x $\rq\rq denotes  operation ${A}^{(132)}$.

\begin{table}  [h!] 
\caption{Units in quasigroup that is a  generalized derivative} \label{TABLE_two_thre}
\large{
\begin{center}
\begin{tabular}{c c c }
\begin{tabular}{|c| |c |c |c| }
\hline
$(L_a, L_a, \varepsilon)$  & $f$  & $e$ & $s$ \\
\hline
$xy$  & $+$  & $-$ & $-$ \\
$yx$  & $-$  & $+$ & $-$ \\
$x\backslash y$  & $-$  & $-$ & $-$ \\
$y\backslash x$  & $-$  & $-$ & $-$ \\
$y\slash x$  & $-$  & $-$ & $-$ \\
$x\slash y$  & $-$  & $-$ & $-$ \\
\hline
\end{tabular}
&
\begin{tabular}{|c| |c |c |c| }
\hline
$(L_a, L^{-1}_a, \varepsilon)$  & $f$  & $e$ & $s$ \\
\hline
$xy$  & $-$  & $-$ & $-$ \\
$yx$  & $-$  & $+$ & $-$ \\
$x\backslash y$  & $+$  & $-$ & $-$ \\
$y\backslash x$  & $-$  & $-$ & $-$ \\
$y\slash x$  & $-$  & $-$ & $-$ \\
$x\slash y$  & $-$  & $-$ & $-$ \\
\hline
\end{tabular}
&
\begin{tabular}{|c| |c |c |c| }
\hline
$(L_a, R_a, \varepsilon)$  & $f$  & $e$ & $s$ \\
\hline
$xy$  & $-$  & $-$ & $-$ \\
$yx$  & $+$  & $+$ & $-$ \\
$x\backslash y$  & $-$  & $-$ & $-$ \\
$y\backslash x$  & $-$  & $-$ & $-$ \\
$y\slash x$  & $-$  & $-$ & $-$ \\
$x\slash y$  & $-$  & $-$ & $-$ \\
\hline
\end{tabular}
\end{tabular}\\

\begin{tabular}{c c c }
\begin{tabular}{|c| |c |c |c| }
\hline
$(L_a, R^{-1}_a, \varepsilon)$  & $f$  & $e$ & $s$ \\
\hline
$xy$  & $-$  & $-$ & $-$ \\
$yx$  & $-$  & $+$ & $-$ \\
$x\backslash y$  & $-$  & $-$ & $-$ \\
$y\backslash x$  & $-$  & $-$ & $-$ \\
$y\slash x$  & $+$  & $-$ & $-$ \\
$x\slash y$  & $-$  & $-$ & $-$ \\
\hline
\end{tabular}
&
\begin{tabular}{|c| |c |c |c| }
\hline
$(L_a, P_a, \varepsilon)$  & $f$  & $e$ & $s$ \\
\hline
$xy$  & $-$  & $-$ & $-$ \\
$yx$  & $-$  & $+$ & $-$ \\
$x\backslash y$  & $-$  & $-$ & $-$ \\
$y\backslash x$  & $+$  & $-$ & $-$ \\
$y\slash x$  & $-$  & $-$ & $-$ \\
$x\slash y$  & $-$  & $-$ & $-$ \\
\hline
\end{tabular}
&
\begin{tabular}{|c| |c |c |c| }
\hline
$(L_a, P^{-1}_a, \varepsilon)$  & $f$  & $e$ & $s$ \\
\hline
$xy$  & $-$  & $-$ & $-$ \\
$yx$  & $-$  & $+$ & $-$ \\
$x\backslash y$  & $-$  & $-$ & $-$ \\
$y\backslash x$  & $-$  & $-$ & $-$ \\
$y\slash x$  & $-$  & $-$ & $-$ \\
$x\slash y$  & $+$  & $-$ & $-$ \\
\hline
\end{tabular}
\end{tabular}\\

\begin{tabular}{c c c }
\begin{tabular}{|c| |c |c |c| }
\hline
$(L_a, \varepsilon, L_a)$  & $f$  & $e$ & $s$ \\
\hline
$xy$  & $-$  & $-$ & $-$ \\
$yx$  & $-$  & $-$ & $-$ \\
$x\backslash y$  & $+$  & $-$ & $-$ \\
$y\backslash x$  & $-$  & $-$ & $-$ \\
$y\slash x$  & $-$  & $-$ & $+$ \\
$x\slash y$  & $-$  & $-$ & $-$ \\
\hline
\end{tabular}
&
\begin{tabular}{|c| |c |c |c| }
\hline
$(L_a, \varepsilon,  L^{-1}_a)$  & $f$  & $e$ & $s$ \\
\hline
$xy$  & $+$  & $-$ & $-$ \\
$yx$  & $-$  & $-$ & $-$ \\
$x\backslash y$  & $-$  & $-$ & $-$ \\
$y\backslash x$  & $-$  & $-$ & $-$ \\
$y\slash x$  & $-$  & $-$ & $+$ \\
$x\slash y$  & $-$  & $-$ & $-$ \\
\hline
\end{tabular}
&
\begin{tabular}{|c| |c |c |c| }
\hline
$(L_a, \varepsilon, R_a)$  & $f$  & $e$ & $s$ \\
\hline
$xy$  & $-$  & $-$ & $-$ \\
$yx$  & $-$  & $-$ & $-$ \\
$x\backslash y$  & $-$  & $-$ & $-$ \\
$y\backslash x$  & $-$  & $-$ & $-$ \\
$y\slash x$  & $+$  & $-$ & $+$ \\
$x\slash y$  & $-$  & $-$ & $-$ \\
\hline
\end{tabular}
\end{tabular}

\begin{tabular}{c c c }
\begin{tabular}{|c| |c |c |c| }
\hline
$(L_a, \varepsilon, R^{-1}_a)$  & $f$  & $e$ & $s$ \\
\hline
$xy$  & $-$  & $-$ & $-$ \\
$yx$  & $+$  & $-$ & $-$ \\
$x\backslash y$  & $-$  & $-$ & $-$ \\
$y\backslash x$  & $-$  & $-$ & $-$ \\
$y\slash x$  & $-$  & $-$ & $+$ \\
$x\slash y$  & $-$  & $-$ & $-$ \\
\hline
\end{tabular}
&
\begin{tabular}{|c| |c |c |c| }
\hline
$(L_a, \varepsilon,  P_a)$  & $f$  & $e$ & $s$ \\
\hline
$xy$  & $-$  & $-$ & $-$ \\
$yx$  & $-$  & $-$ & $-$ \\
$x\backslash y$  & $-$  & $-$ & $-$ \\
$y\backslash x$  & $-$  & $-$ & $-$ \\
$y\slash x$  & $-$  & $-$ & $+$ \\
$x\slash y$  & $+$  & $-$ & $-$ \\
\hline
\end{tabular}
&
\begin{tabular}{|c| |c |c |c| }
\hline
$(L_a, \varepsilon, P^{-1}_a)$  & $f$  & $e$ & $s$ \\
\hline
$xy$  & $-$  & $-$ & $-$ \\
$yx$  & $-$  & $-$ & $-$ \\
$x\backslash y$  & $-$  & $-$ & $-$ \\
$y\backslash x$  & $+$  & $-$ & $-$ \\
$y\slash x$  & $-$  & $-$ & $+$ \\
$x\slash y$  & $-$  & $-$ & $-$ \\
\hline
\end{tabular}
\end{tabular}

\end{center}
}
\end{table}

\begin{table}[hptb]
\label{TABLE_thre}
\large{
\begin{center}
\begin{tabular}{c c c }
\begin{tabular}{|c| |c |c |c| }
\hline
$(\varepsilon, L_a, L_a)$  & $f$  & $e$ & $s$ \\
\hline
$xy$  & $-$  & $-$ & $-$ \\
$yx$  & $-$  & $-$ & $-$ \\
$x\backslash y$  & $-$  & $-$ & $-$ \\
$y\backslash x$  & $-$  & $+$ & $-$ \\
$y\slash x$  & $-$  & $-$ & $-$ \\
$x\slash y$  & $-$  & $-$ & $+$ \\
\hline
\end{tabular}
&
\begin{tabular}{|c| |c |c |c| }
\hline
$(\varepsilon, L_a, L^{-1}_a)$  & $f$  & $e$ & $s$ \\
\hline
$xy$  & $-$  & $-$ & $-$ \\
$yx$  & $-$  & $+$ & $-$ \\
$x\backslash y$  & $-$  & $-$ & $-$ \\
$y\backslash x$  & $-$  & $-$ & $-$ \\
$y\slash x$  & $-$  & $-$ & $-$ \\
$x\slash y$  & $-$  & $-$ & $+$ \\
\hline
\end{tabular}
&
\begin{tabular}{|c| |c |c |c| }
\hline
$(\varepsilon, L_a, R_a)$  & $f$  & $e$ & $s$ \\
\hline
$xy$  & $-$  & $-$ & $-$ \\
$yx$  & $-$  & $-$ & $-$ \\
$x\backslash y$  & $-$  & $-$ & $-$ \\
$y\backslash x$  & $-$  & $-$ & $-$ \\
$y\slash x$  & $-$  & $-$ & $-$ \\
$x\slash y$  & $-$  & $+$ & $+$ \\
\hline
\end{tabular}
\end{tabular}\\

\begin{tabular}{c c c }
\begin{tabular}{|c| |c |c |c| }
\hline
$(\varepsilon, L_a, R^{-1}_a)$  & $f$  & $e$ & $s$ \\
\hline
$xy$  & $-$  & $+$ & $-$ \\
$yx$  & $-$  & $-$ & $-$ \\
$x\backslash y$  & $-$  & $-$ & $-$ \\
$y\backslash x$  & $-$  & $-$ & $-$ \\
$y\slash x$  & $-$  & $-$ & $-$ \\
$x\slash y$  & $-$  & $-$ & $+$ \\
\hline
\end{tabular}
&
\begin{tabular}{|c| |c |c |c| }
\hline
$(\varepsilon, L_a, P_a)$  & $f$  & $e$ & $s$ \\
\hline
$xy$  & $-$  & $-$ & $-$ \\
$yx$  & $-$  & $-$ & $-$ \\
$x\backslash y$  & $-$  & $-$ & $-$ \\
$y\backslash x$  & $-$  & $-$ & $-$ \\
$y\slash x$  & $-$  & $+$ & $-$ \\
$x\slash y$  & $-$  & $-$ & $+$ \\
\hline
\end{tabular}
&
\begin{tabular}{|c| |c |c |c| }
\hline
$(\varepsilon, L_a, P^{-1}_a)$  & $f$  & $e$ & $s$ \\
\hline
$xy$  & $-$  & $-$ & $-$ \\
$yx$  & $-$  & $-$ & $-$ \\
$x\backslash y$  & $-$  & $+$ & $-$ \\
$y\backslash x$  & $-$  & $-$ & $-$ \\
$y\slash x$  & $-$  & $-$ & $-$ \\
$x\slash y$  & $-$  & $-$ & $+$ \\
\hline
\end{tabular}
\end{tabular}\\

\begin{tabular}{c c c }
\begin{tabular}{|c| |c |c |c| }
\hline
$(L^{-1}_a, L_a, \varepsilon)$  & $f$  & $e$ & $s$ \\
\hline
$xy$  & $+$  & $-$ & $-$ \\
$yx$  & $-$  & $-$ & $-$ \\
$x\backslash y$  & $-$  & $-$ & $-$ \\
$y\backslash x$  & $-$  & $+$ & $-$ \\
$y\slash x$  & $-$  & $-$ & $-$ \\
$x\slash y$  & $-$  & $-$ & $-$ \\
\hline
\end{tabular}
&
\begin{tabular}{|c| |c |c |c| }
\hline
$(L^{-1}_a,  L^{-1}_a, \varepsilon)$  & $f$  & $e$ & $s$ \\
\hline
$xy$  & $-$  & $-$ & $-$ \\
$yx$  & $-$  & $-$ & $-$ \\
$x\backslash y$  & $+$  & $-$ & $-$ \\
$y\backslash x$  & $-$  & $+$ & $-$ \\
$y\slash x$  & $-$  & $-$ & $-$ \\
$x\slash y$  & $-$  & $-$ & $-$ \\
\hline
\end{tabular}
&
\begin{tabular}{|c| |c |c |c| }
\hline
$(L^{-1}_a, R_a,  \varepsilon)$  & $f$  & $e$ & $s$ \\
\hline
$xy$  & $-$  & $-$ & $-$ \\
$yx$  & $+$  & $-$ & $-$ \\
$x\backslash y$  & $-$  & $-$ & $-$ \\
$y\backslash x$  & $-$  & $+$ & $-$ \\
$y\slash x$  & $-$  & $-$ & $-$ \\
$x\slash y$  & $-$  & $-$ & $-$ \\
\hline
\end{tabular}
\end{tabular}\\

\begin{tabular}{c c c }
\begin{tabular}{|c| |c |c |c| }
\hline
$(L^{-1}_a, R^{-1}_a, \varepsilon)$  & $f$  & $e$ & $s$ \\
\hline
$xy$  & $-$  & $-$ & $-$ \\
$yx$  & $-$  & $-$ & $-$ \\
$x\backslash y$  & $-$  & $-$ & $-$ \\
$y\backslash x$  & $-$  & $+$ & $-$ \\
$y\slash x$  & $+$  & $-$ & $-$ \\
$x\slash y$  & $-$  & $-$ & $-$ \\
\hline
\end{tabular}
&
\begin{tabular}{|c| |c |c |c| }
\hline
$(L^{-1}_a,  P_a, \varepsilon)$  & $f$  & $e$ & $s$ \\
\hline
$xy$  & $-$  & $-$ & $-$ \\
$yx$  & $-$  & $-$ & $-$ \\
$x\backslash y$  & $-$  & $-$ & $-$ \\
$y\backslash x$  & $+$  & $+$ & $-$ \\
$y\slash x$  & $-$  & $-$ & $-$ \\
$x\slash y$  & $-$  & $-$ & $-$ \\
\hline
\end{tabular}
&
\begin{tabular}{|c| |c |c |c| }
\hline
$(L^{-1}_a,  P^{-1}_a, \varepsilon)$  & $f$  & $e$ & $s$ \\
\hline
$xy$  & $-$  & $-$ & $-$ \\
$yx$  & $-$  & $-$ & $-$ \\
$x\backslash y$  & $-$  & $-$ & $-$ \\
$y\backslash x$  & $-$  & $+$ & $-$ \\
$y\slash x$  & $-$  & $-$ & $-$ \\
$x\slash y$  & $+$  & $-$ & $-$ \\
\hline
\end{tabular}
\end{tabular}

\end{center}
}
\end{table}

\begin{table}[hptb]
\label{TABLE_thre}
\large{
\begin{center}
\begin{tabular}{c c c }
\begin{tabular}{|c| |c |c |c| }
\hline
$(L^{-1}_a, \varepsilon,  L_a)$  & $f$  & $e$ & $s$ \\
\hline
$xy$  & $-$  & $-$ & $-$ \\
$yx$  & $-$  & $-$ & $-$ \\
$x\backslash y$  & $+$  & $-$ & $-$ \\
$y\backslash x$  & $-$  & $-$ & $-$ \\
$y\slash x$  & $-$  & $-$ & $-$ \\
$x\slash y$  & $-$  & $-$ & $+$ \\
\hline
\end{tabular}
&
\begin{tabular}{|c| |c |c |c| }
\hline
$(L^{-1}_a, \varepsilon, L^{-1}_a)$  & $f$  & $e$ & $s$ \\
\hline
$xy$  & $+$  & $-$ & $-$ \\
$yx$  & $-$  & $-$ & $-$ \\
$x\backslash y$  & $-$  & $-$ & $-$ \\
$y\backslash x$  & $-$  & $-$ & $-$ \\
$y\slash x$  & $-$  & $-$ & $-$ \\
$x\slash y$  & $-$  & $-$ & $+$ \\
\hline
\end{tabular}
&
\begin{tabular}{|c| |c |c |c| }
\hline
$(L^{-1}_a, \varepsilon, R_a)$  & $f$  & $e$ & $s$ \\
\hline
$xy$  & $-$  & $-$ & $-$ \\
$yx$  & $-$  & $-$ & $-$ \\
$x\backslash y$  & $-$  & $-$ & $-$ \\
$y\backslash x$  & $-$  & $-$ & $-$ \\
$y\slash x$  & $+$  & $-$ & $-$ \\
$x\slash y$  & $-$  & $-$ & $+$ \\
\hline
\end{tabular}
\end{tabular}\\

\begin{tabular}{c c c }
\begin{tabular}{|c| |c |c |c| }
\hline
$(L^{-1}_a, \varepsilon, R^{-1}_a)$  & $f$  & $e$ & $s$ \\
\hline
$xy$  & $-$  & $-$ & $-$ \\
$yx$  & $+$  & $-$ & $-$ \\
$x\backslash y$  & $-$  & $-$ & $-$ \\
$y\backslash x$  & $-$  & $-$ & $-$ \\
$y\slash x$  & $-$  & $-$ & $-$ \\
$x\slash y$  & $-$  & $-$ & $+$ \\
\hline
\end{tabular}
&
\begin{tabular}{|c| |c |c |c| }
\hline
$(L^{-1}_a, \varepsilon,  P_a)$  & $f$  & $e$ & $s$ \\
\hline
$xy$  & $-$  & $-$ & $-$ \\
$yx$  & $-$  & $-$ & $-$ \\
$x\backslash y$  & $-$  & $-$ & $-$ \\
$y\backslash x$  & $-$  & $-$ & $-$ \\
$y\slash x$  & $-$  & $-$ & $-$ \\
$x\slash y$  & $+$  & $-$ & $+$ \\
\hline
\end{tabular}
&
\begin{tabular}{|c| |c |c |c| }
\hline
$(L^{-1}_a, \varepsilon, P^{-1}_a)$  & $f$  & $e$ & $s$ \\
\hline
$xy$  & $-$  & $-$ & $-$ \\
$yx$  & $-$  & $-$ & $-$ \\
$x\backslash y$  & $-$  & $-$ & $-$ \\
$y\backslash x$  & $+$  & $-$ & $-$ \\
$y\slash x$  & $-$  & $-$ & $-$ \\
$x\slash y$  & $-$  & $-$ & $+$ \\
\hline
\end{tabular}
\end{tabular}\\

\begin{tabular}{c c c }
\begin{tabular}{|c| |c |c |c| }
\hline
$( \varepsilon, L^{-1}_a, L_a)$  & $f$  & $e$ & $s$ \\
\hline
$xy$  & $-$  & $-$ & $-$ \\
$yx$  & $-$  & $-$ & $-$ \\
$x\backslash y$  & $-$  & $-$ & $-$ \\
$y\backslash x$  & $-$  & $+$ & $-$ \\
$y\slash x$  & $-$  & $-$ & $+$ \\
$x\slash y$  & $-$  & $-$ & $-$ \\
\hline
\end{tabular}
&
\begin{tabular}{|c| |c |c |c| }
\hline
$( \varepsilon, L^{-1}_a,  L^{-1}_a)$  & $f$  & $e$ & $s$ \\
\hline
$xy$  & $-$  & $-$ & $-$ \\
$yx$  & $-$  & $+$ & $-$ \\
$x\backslash y$  & $-$  & $-$ & $-$ \\
$y\backslash x$  & $-$  & $-$ & $-$ \\
$y\slash x$  & $-$  & $-$ & $+$ \\
$x\slash y$  & $-$  & $-$ & $-$ \\
\hline
\end{tabular}
&
\begin{tabular}{|c| |c |c |c| }
\hline
$(\varepsilon, L^{-1}_a, R_a)$  & $f$  & $e$ & $s$ \\
\hline
$xy$  & $-$  & $-$ & $-$ \\
$yx$  & $-$  & $-$ & $-$ \\
$x\backslash y$  & $-$  & $-$ & $-$ \\
$y\backslash x$  & $-$  & $-$ & $-$ \\
$y\slash x$  & $-$  & $-$ & $+$ \\
$x\slash y$  & $-$  & $+$ & $-$ \\
\hline
\end{tabular}
\end{tabular}\\

\begin{tabular}{c c c }
\begin{tabular}{|c| |c |c |c| }
\hline
$(\varepsilon, L^{-1}_a, R^{-1}_a)$  & $f$  & $e$ & $s$ \\
\hline
$xy$  & $-$  & $+$ & $-$ \\
$yx$  & $-$  & $-$ & $-$ \\
$x\backslash y$  & $-$  & $-$ & $-$ \\
$y\backslash x$  & $-$  & $-$ & $-$ \\
$y\slash x$  & $-$  & $-$ & $+$ \\
$x\slash y$  & $-$  & $-$ & $-$ \\
\hline
\end{tabular}
&
\begin{tabular}{|c| |c |c |c| }
\hline
$( \varepsilon, L^{-1}_a,  P_a)$  & $f$  & $e$ & $s$ \\
\hline
$xy$  & $-$  & $-$ & $-$ \\
$yx$  & $-$  & $-$ & $-$ \\
$x\backslash y$  & $-$  & $-$ & $-$ \\
$y\backslash x$  & $-$  & $-$ & $-$ \\
$y\slash x$  & $-$  & $+$ & $+$ \\
$x\slash y$  & $-$  & $-$ & $-$ \\
\hline
\end{tabular}
&
\begin{tabular}{|c| |c |c |c| }
\hline
$( \varepsilon, L^{-1}_a,  P^{-1}_a)$  & $f$  & $e$ & $s$ \\
\hline
$xy$  & $-$  & $-$ & $-$ \\
$yx$  & $-$  & $-$ & $-$ \\
$x\backslash y$  & $-$  & $+$ & $-$ \\
$y\backslash x$  & $-$  & $-$ & $-$ \\
$y\slash x$  & $-$  & $-$ & $+$ \\
$x\slash y$  & $-$  & $-$ & $-$ \\
\hline
\end{tabular}
\end{tabular}

\end{center}
}
\end{table}

\begin{table}[hptb]
\label{TABLE_thre}
\large{
\begin{center}
\begin{tabular}{c c c }
\begin{tabular}{|c| |c |c |c| }
\hline
$(R_a,  L_a, \varepsilon)$  & $f$  & $e$ & $s$ \\
\hline
$xy$  & $+$  & $+$ & $-$ \\
$yx$  & $-$  & $-$ & $-$ \\
$x\backslash y$  & $-$  & $-$ & $-$ \\
$y\backslash x$  & $-$  & $-$ & $-$ \\
$y\slash x$  & $-$  & $-$ & $-$ \\
$x\slash y$  & $-$  & $-$ & $-$ \\
\hline
\end{tabular}
&
\begin{tabular}{|c| |c |c |c| }
\hline
$(R_a, L^{-1}_a, \varepsilon)$  & $f$  & $e$ & $s$ \\
\hline
$xy$  & $-$  & $+$ & $-$ \\
$yx$  & $-$  & $-$ & $-$ \\
$x\backslash y$  & $+$  & $-$ & $-$ \\
$y\backslash x$  & $-$  & $-$ & $-$ \\
$y\slash x$  & $-$  & $-$ & $-$ \\
$x\slash y$  & $-$  & $-$ & $-$ \\
\hline
\end{tabular}
&
\begin{tabular}{|c| |c |c |c| }
\hline
$(R_a, R_a, \varepsilon)$  & $f$  & $e$ & $s$ \\
\hline
$xy$  & $-$  & $+$ & $-$ \\
$yx$  & $+$  & $-$ & $-$ \\
$x\backslash y$  & $-$  & $-$ & $-$ \\
$y\backslash x$  & $-$  & $-$ & $-$ \\
$y\slash x$  & $-$  & $-$ & $-$ \\
$x\slash y$  & $-$  & $-$ & $-$ \\
\hline
\end{tabular}
\end{tabular}\\

\begin{tabular}{c c c }
\begin{tabular}{|c| |c |c |c| }
\hline
$(R_a, R^{-1}_a, \varepsilon)$  & $f$  & $e$ & $s$ \\
\hline
$xy$  & $-$  & $+$ & $-$ \\
$yx$  & $-$  & $-$ & $-$ \\
$x\backslash y$  & $-$  & $-$ & $-$ \\
$y\backslash x$  & $-$  & $-$ & $-$ \\
$y\slash x$  & $+$  & $-$ & $-$ \\
$x\slash y$  & $-$  & $-$ & $-$ \\
\hline
\end{tabular}
&
\begin{tabular}{|c| |c |c |c| }
\hline
$(R_a,  P_a, \varepsilon)$  & $f$  & $e$ & $s$ \\
\hline
$xy$  & $-$  & $+$ & $-$ \\
$yx$  & $-$  & $-$ & $-$ \\
$x\backslash y$  & $-$  & $-$ & $-$ \\
$y\backslash x$  & $+$  & $-$ & $-$ \\
$y\slash x$  & $-$  & $-$ & $-$ \\
$x\slash y$  & $-$  & $-$ & $-$ \\
\hline
\end{tabular}
&
\begin{tabular}{|c| |c |c |c| }
\hline
$(R_a, P^{-1}_a, \varepsilon)$  & $f$  & $e$ & $s$ \\
\hline
$xy$  & $-$  & $+$ & $-$ \\
$yx$  & $-$  & $-$ & $-$ \\
$x\backslash y$  & $-$  & $-$ & $-$ \\
$y\backslash x$  & $-$  & $-$ & $-$ \\
$y\slash x$  & $-$  & $-$ & $-$ \\
$x\slash y$  & $+$  & $-$ & $-$ \\
\hline
\end{tabular}
\end{tabular}\\

\begin{tabular}{c c c }
\begin{tabular}{|c| |c |c |c| }
\hline
$(R_a, \varepsilon,  L_a)$  & $f$  & $e$ & $s$ \\
\hline
$xy$  & $-$  & $-$ & $-$ \\
$yx$  & $-$  & $-$ & $-$ \\
$x\backslash y$  & $+$  & $-$ & $+$ \\
$y\backslash x$  & $-$  & $-$ & $-$ \\
$y\slash x$  & $-$  & $-$ & $-$ \\
$x\slash y$  & $-$  & $-$ & $-$ \\
\hline
\end{tabular}
&
\begin{tabular}{|c| |c |c |c| }
\hline
$(R_a, \varepsilon,   L^{-1}_a)$  & $f$  & $e$ & $s$ \\
\hline
$xy$  & $+$  & $-$ & $-$ \\
$yx$  & $-$  & $-$ & $-$ \\
$x\backslash y$  & $-$  & $-$ & $+$ \\
$y\backslash x$  & $-$  & $-$ & $-$ \\
$y\slash x$  & $-$  & $-$ & $-$ \\
$x\slash y$  & $-$  & $-$ & $-$ \\
\hline
\end{tabular}
&
\begin{tabular}{|c| |c |c |c| }
\hline
$(R_a, \varepsilon, R_a)$  & $f$  & $e$ & $s$ \\
\hline
$xy$  & $-$  & $-$ & $-$ \\
$yx$  & $-$  & $-$ & $-$ \\
$x\backslash y$  & $-$  & $-$ & $+$ \\
$y\backslash x$  & $-$  & $-$ & $-$ \\
$y\slash x$  & $+$  & $-$ & $-$ \\
$x\slash y$  & $-$  & $-$ & $-$ \\
\hline
\end{tabular}
\end{tabular}\\

\begin{tabular}{c c c }
\begin{tabular}{|c| |c |c |c| }
\hline
$(R_a, \varepsilon,  R^{-1}_a)$  & $f$  & $e$ & $s$ \\
\hline
$xy$  & $-$  & $-$ & $-$ \\
$yx$  & $+$  & $-$ & $-$ \\
$x\backslash y$  & $-$  & $-$ & $+$ \\
$y\backslash x$  & $-$  & $-$ & $-$ \\
$y\slash x$  & $-$  & $-$ & $-$ \\
$x\slash y$  & $-$  & $-$ & $-$ \\
\hline
\end{tabular}
&
\begin{tabular}{|c| |c |c |c| }
\hline
$(R_a, \varepsilon,   P_a)$  & $f$  & $e$ & $s$ \\
\hline
$xy$  & $-$  & $-$ & $-$ \\
$yx$  & $-$  & $-$ & $-$ \\
$x\backslash y$  & $-$  & $-$ & $+$ \\
$y\backslash x$  & $-$  & $-$ & $-$ \\
$y\slash x$  & $-$  & $-$ & $-$ \\
$x\slash y$  & $+$  & $-$ & $-$ \\
\hline
\end{tabular}
&
\begin{tabular}{|c| |c |c |c| }
\hline
$(R_a, \varepsilon,   P^{-1}_a)$  & $f$  & $e$ & $s$ \\
\hline
$xy$  & $-$  & $-$ & $-$ \\
$yx$  & $-$  & $-$ & $-$ \\
$x\backslash y$  & $-$  & $-$ & $+$ \\
$y\backslash x$  & $+$  & $-$ & $-$ \\
$y\slash x$  & $-$  & $-$ & $-$ \\
$x\slash y$  & $-$  & $-$ & $-$ \\
\hline
\end{tabular}
\end{tabular}

\end{center}
}
\end{table}

\newpage

\begin{table}[hptb]
\label{TABLE_thre}
\large{
\begin{center}
\begin{tabular}{c c c }
\begin{tabular}{|c| |c |c |c| }
\hline
$(\varepsilon, R_a, L_a)$  & $f$  & $e$ & $s$ \\
\hline
$xy$  & $-$  & $=$ & $-$ \\
$yx$  & $-$  & $-$ & $-$ \\
$x\backslash y$  & $-$  & $-$ & $-$ \\
$y\backslash x$  & $-$  & $+$ & $+$ \\
$y\slash x$  & $-$  & $-$ & $-$ \\
$x\slash y$  & $-$  & $-$ & $-$ \\
\hline
\end{tabular}
&
\begin{tabular}{|c| |c |c |c| }
\hline
$(\varepsilon, R_a, L^{-1}_a)$  & $f$  & $e$ & $s$ \\
\hline
$xy$  & $-$  & $-$ & $-$ \\
$yx$  & $-$  & $+$ & $-$ \\
$x\backslash y$  & $-$  & $-$ & $-$ \\
$y\backslash x$  & $-$  & $-$ & $+$ \\
$y\slash x$  & $-$  & $-$ & $-$ \\
$x\slash y$  & $-$  & $-$ & $-$ \\
\hline
\end{tabular}
&
\begin{tabular}{|c| |c |c |c| }
\hline
$(\varepsilon, R_a, R_a)$  & $f$  & $e$ & $s$ \\
\hline
$xy$  & $-$  & $-$ & $-$ \\
$yx$  & $-$  & $-$ & $-$ \\
$x\backslash y$  & $-$  & $-$ & $-$ \\
$y\backslash x$  & $-$  & $-$ & $+$ \\
$y\slash x$  & $-$  & $-$ & $-$ \\
$x\slash y$  & $-$  & $+$ & $-$ \\
\hline
\end{tabular}
\end{tabular}\\

\begin{tabular}{c c c }
\begin{tabular}{|c| |c |c |c| }
\hline
$(\varepsilon, R_a, R^{-1}_a)$  & $f$  & $e$ & $s$ \\
\hline
$xy$  & $-$  & $+$ & $-$ \\
$yx$  & $-$  & $-$ & $-$ \\
$x\backslash y$  & $-$  & $-$ & $-$ \\
$y\backslash x$  & $-$  & $-$ & $+$ \\
$y\slash x$  & $-$  & $-$ & $-$ \\
$x\slash y$  & $-$  & $-$ & $-$ \\
\hline
\end{tabular}
&
\begin{tabular}{|c| |c |c |c| }
\hline
$(\varepsilon, R_a, P_a)$  & $f$  & $e$ & $s$ \\
\hline
$xy$  & $-$  & $-$ & $-$ \\
$yx$  & $-$  & $-$ & $-$ \\
$x\backslash y$  & $-$  & $-$ & $-$ \\
$y\backslash x$  & $-$  & $-$ & $+$ \\
$y\slash x$  & $-$  & $+$ & $-$ \\
$x\slash y$  & $-$  & $-$ & $-$ \\
\hline
\end{tabular}
&
\begin{tabular}{|c| |c |c |c| }
\hline
$(\varepsilon, R_a, P^{-1}_a)$  & $f$  & $e$ & $s$ \\
\hline
$xy$  & $-$  & $-$ & $-$ \\
$yx$  & $-$  & $-$ & $-$ \\
$x\backslash y$  & $-$  & $+$ & $-$ \\
$y\backslash x$  & $-$  & $-$ & $+$ \\
$y\slash x$  & $-$  & $-$ & $-$ \\
$x\slash y$  & $-$  & $-$ & $-$ \\
\hline
\end{tabular}
\end{tabular}\\

\begin{tabular}{c c c }
\begin{tabular}{|c| |c |c |c| }
\hline
$(R^{-1}_a, L_a, \varepsilon)$  & $f$  & $e$ & $s$ \\
\hline
$xy$  & $+$  & $-$ & $-$ \\
$yx$  & $-$  & $-$ & $-$ \\
$x\backslash y$  & $-$  & $-$ & $-$ \\
$y\backslash x$  & $-$  & $-$ & $-$ \\
$y\slash x$  & $-$  & $-$ & $-$ \\
$x\slash y$  & $-$  & $+$ & $-$ \\
\hline
\end{tabular}
&
\begin{tabular}{|c| |c |c |c| }
\hline
$(R^{-1}_a,  L^{-1}_a, \varepsilon)$  & $f$  & $e$ & $s$ \\
\hline
$xy$  & $-$  & $-$ & $-$ \\
$yx$  & $-$  & $-$ & $-$ \\
$x\backslash y$  & $+$  & $-$ & $-$ \\
$y\backslash x$  & $-$  & $-$ & $-$ \\
$y\slash x$  & $-$  & $-$ & $-$ \\
$x\slash y$  & $-$  & $+$ & $-$ \\
\hline
\end{tabular}
&
\begin{tabular}{|c| |c |c |c| }
\hline
$(R^{-1}_a, R_a,  \varepsilon)$  & $f$  & $e$ & $s$ \\
\hline
$xy$  & $-$  & $-$ & $-$ \\
$yx$  & $+$  & $-$ & $-$ \\
$x\backslash y$  & $-$  & $-$ & $-$ \\
$y\backslash x$  & $-$  & $-$ & $-$ \\
$y\slash x$  & $-$  & $-$ & $-$ \\
$x\slash y$  & $-$  & $+$ & $-$ \\
\hline
\end{tabular}
\end{tabular}\\

\begin{tabular}{c c c }
\begin{tabular}{|c| |c |c |c| }
\hline
$(R^{-1}_a, R^{-1}_a, \varepsilon)$  & $f$  & $e$ & $s$ \\
\hline
$xy$  & $-$  & $-$ & $-$ \\
$yx$  & $-$  & $-$ & $-$ \\
$x\backslash y$  & $-$  & $-$ & $-$ \\
$y\backslash x$  & $-$  & $-$ & $-$ \\
$y\slash x$  & $+$  & $-$ & $-$ \\
$x\slash y$  & $-$  & $+$ & $-$ \\
\hline
\end{tabular}
&
\begin{tabular}{|c| |c |c |c| }
\hline
$(R^{-1}_a,  P_a, \varepsilon)$  & $f$  & $e$ & $s$ \\
\hline
$xy$  & $-$  & $-$ & $-$ \\
$yx$  & $-$  & $-$ & $-$ \\
$x\backslash y$  & $-$  & $-$ & $-$ \\
$y\backslash x$  & $+$  & $-$ & $-$ \\
$y\slash x$  & $-$  & $-$ & $-$ \\
$x\slash y$  & $-$  & $+$ & $-$ \\
\hline
\end{tabular}
&
\begin{tabular}{|c| |c |c |c| }
\hline
$(R^{-1}_a,  P^{-1}_a, \varepsilon)$  & $f$  & $e$ & $s$ \\
\hline
$xy$  & $-$  & $-$ & $-$ \\
$yx$  & $-$  & $-$ & $-$ \\
$x\backslash y$  & $-$  & $-$ & $-$ \\
$y\backslash x$  & $-$  & $-$ & $-$ \\
$y\slash x$  & $-$  & $-$ & $-$ \\
$x\slash y$  & $+$  & $+$ & $-$ \\
\hline
\end{tabular}
\end{tabular}

\end{center}
}
\end{table}

\begin{table}[hptb]
\label{TABLE_thre}
\large{
\begin{center}
\begin{tabular}{c c c }
\begin{tabular}{|c| |c |c |c| }
\hline
$(R^{-1}_a, \varepsilon,  L_a)$  & $f$  & $e$ & $s$ \\
\hline
$xy$  & $-$  & $-$ & $-$ \\
$yx$  & $-$  & $-$ & $-$ \\
$x\backslash y$  & $+$  & $-$ & $-$ \\
$y\backslash x$  & $-$  & $-$ & $+$ \\
$y\slash x$  & $-$  & $-$ & $-$ \\
$x\slash y$  & $-$  & $-$ & $-$ \\
\hline
\end{tabular}
&
\begin{tabular}{|c| |c |c |c| }
\hline
$(R^{-1}_a, \varepsilon, L^{-1}_a)$  & $f$  & $e$ & $s$ \\
\hline
$xy$  & $+$  & $-$ & $-$ \\
$yx$  & $-$  & $-$ & $-$ \\
$x\backslash y$  & $-$  & $-$ & $-$ \\
$y\backslash x$  & $-$  & $-$ & $+$ \\
$y\slash x$  & $-$  & $-$ & $-$ \\
$x\slash y$  & $-$  & $-$ & $-$ \\
\hline
\end{tabular}
&
\begin{tabular}{|c| |c |c |c| }
\hline
$(R^{-1}_a, \varepsilon, R_a)$  & $f$  & $e$ & $s$ \\
\hline
$xy$  & $-$  & $-$ & $-$ \\
$yx$  & $-$  & $-$ & $-$ \\
$x\backslash y$  & $-$  & $-$ & $-$ \\
$y\backslash x$  & $-$  & $-$ & $+$ \\
$y\slash x$  & $+$  & $-$ & $-$ \\
$x\slash y$  & $-$  & $-$ & $-$ \\
\hline
\end{tabular}
\end{tabular}\\

\begin{tabular}{c c c }
\begin{tabular}{|c| |c |c |c| }
\hline
$(R^{-1}_a, \varepsilon, R^{-1}_a)$  & $f$  & $e$ & $s$ \\
\hline
$xy$  & $-$  & $-$ & $-$ \\
$yx$  & $+$  & $-$ & $-$ \\
$x\backslash y$  & $-$  & $-$ & $-$ \\
$y\backslash x$  & $-$  & $-$ & $+$ \\
$y\slash x$  & $-$  & $-$ & $-$ \\
$x\slash y$  & $-$  & $-$ & $-$ \\
\hline
\end{tabular}
&
\begin{tabular}{|c| |c |c |c| }
\hline
$(R^{-1}_a, \varepsilon,  P_a)$  & $f$  & $e$ & $s$ \\
\hline
$xy$  & $-$  & $-$ & $-$ \\
$yx$  & $-$  & $-$ & $-$ \\
$x\backslash y$  & $-$  & $-$ & $-$ \\
$y\backslash x$  & $-$  & $-$ & $+$ \\
$y\slash x$  & $-$  & $-$ & $-$ \\
$x\slash y$  & $+$  & $-$ & $-$ \\
\hline
\end{tabular}
&
\begin{tabular}{|c| |c |c |c| }
\hline
$(R^{-1}_a, \varepsilon, P^{-1}_a)$  & $f$  & $e$ & $s$ \\
\hline
$xy$  & $-$  & $-$ & $-$ \\
$yx$  & $-$  & $-$ & $-$ \\
$x\backslash y$  & $-$  & $-$ & $-$ \\
$y\backslash x$  & $+$  & $-$ & $+$ \\
$y\slash x$  & $-$  & $-$ & $-$ \\
$x\slash y$  & $-$  & $-$ & $-$ \\
\hline
\end{tabular}
\end{tabular}\\

\begin{tabular}{c c c }
\begin{tabular}{|c| |c |c |c| }
\hline
$( \varepsilon, R^{-1}_a, L_a)$  & $f$  & $e$ & $s$ \\
\hline
$xy$  & $-$  & $-$ & $-$ \\
$yx$  & $-$  & $-$ & $-$ \\
$x\backslash y$  & $-$  & $-$ & $+$ \\
$y\backslash x$  & $-$  & $+$ & $-$ \\
$y\slash x$  & $-$  & $-$ & $-$ \\
$x\slash y$  & $-$  & $-$ & $-$ \\
\hline
\end{tabular}
&
\begin{tabular}{|c| |c |c |c| }
\hline
$( \varepsilon, R^{-1}_a,  L^{-1}_a)$  & $f$  & $e$ & $s$ \\
\hline
$xy$  & $-$  & $-$ & $-$ \\
$yx$  & $-$  & $+$ & $-$ \\
$x\backslash y$  & $-$  & $-$ & $+$ \\
$y\backslash x$  & $-$  & $-$ & $-$ \\
$y\slash x$  & $-$  & $-$ & $-$ \\
$x\slash y$  & $-$  & $-$ & $-$ \\
\hline
\end{tabular}
&
\begin{tabular}{|c| |c |c |c| }
\hline
$(\varepsilon, R^{-1}_a, R_a)$  & $f$  & $e$ & $s$ \\
\hline
$xy$  & $-$  & $-$ & $-$ \\
$yx$  & $-$  & $-$ & $-$ \\
$x\backslash y$  & $-$  & $-$ & $+$ \\
$y\backslash x$  & $-$  & $-$ & $-$ \\
$y\slash x$  & $-$  & $-$ & $-$ \\
$x\slash y$  & $-$  & $+$ & $-$ \\
\hline
\end{tabular}
\end{tabular}\\

\begin{tabular}{c c c }
\begin{tabular}{|c| |c |c |c| }
\hline
$(\varepsilon, R^{-1}_a, R^{-1}_a)$  & $f$  & $e$ & $s$ \\
\hline
$xy$  & $-$  & $+$ & $-$ \\
$yx$  & $-$  & $-$ & $-$ \\
$x\backslash y$  & $-$  & $-$ & $+$ \\
$y\backslash x$  & $-$  & $-$ & $-$ \\
$y\slash x$  & $-$  & $-$ & $-$ \\
$x\slash y$  & $-$  & $-$ & $-$ \\
\hline
\end{tabular}
&
\begin{tabular}{|c| |c |c |c| }
\hline
$( \varepsilon, R^{-1}_a,  P_a)$  & $f$  & $e$ & $s$ \\
\hline
$xy$  & $-$  & $-$ & $-$ \\
$yx$  & $-$  & $-$ & $-$ \\
$x\backslash y$  & $-$  & $-$ & $+$ \\
$y\backslash x$  & $-$  & $-$ & $-$ \\
$y\slash x$  & $-$  & $+$ & $-$ \\
$x\slash y$  & $-$  & $-$ & $-$ \\
\hline
\end{tabular}
&
\begin{tabular}{|c| |c |c |c| }
\hline
$( \varepsilon, R^{-1}_a,  P^{-1}_a)$  & $f$  & $e$ & $s$ \\
\hline
$xy$  & $-$  & $-$ & $-$ \\
$yx$  & $-$  & $-$ & $-$ \\
$x\backslash y$  & $-$  & $+$ & $+$ \\
$y\backslash x$  & $-$  & $-$ & $-$ \\
$y\slash x$  & $-$  & $-$ & $-$ \\
$x\slash y$  & $-$  & $-$ & $-$ \\
\hline
\end{tabular}
\end{tabular}

\end{center}
}
\end{table}

\begin{table}[hptb]
\label{TABLE_thre}
\large{
\begin{center}
\begin{tabular}{c c c }
\begin{tabular}{|c| |c |c |c| }
\hline
$(P_a,  L_a, \varepsilon)$  & $f$  & $e$ & $s$ \\
\hline
$xy$  & $+$  & $-$ & $-$ \\
$yx$  & $-$  & $-$ & $-$ \\
$x\backslash y$  & $-$  & $+$ & $-$ \\
$y\backslash x$  & $-$  & $-$ & $-$ \\
$y\slash x$  & $-$  & $-$ & $-$ \\
$x\slash y$  & $-$  & $-$ & $-$ \\
\hline
\end{tabular}
&
\begin{tabular}{|c| |c |c |c| }
\hline
$(P_a, L^{-1}_a, \varepsilon)$  & $f$  & $e$ & $s$ \\
\hline
$xy$  & $-$  & $-$ & $-$ \\
$yx$  & $-$  & $-$ & $-$ \\
$x\backslash y$  & $+$  & $+$ & $-$ \\
$y\backslash x$  & $-$  & $-$ & $-$ \\
$y\slash x$  & $-$  & $-$ & $-$ \\
$x\slash y$  & $-$  & $-$ & $-$ \\
\hline
\end{tabular}
&
\begin{tabular}{|c| |c |c |c| }
\hline
$(P_a, R_a, \varepsilon)$  & $f$  & $e$ & $s$ \\
\hline
$xy$  & $-$  & $-$ & $-$ \\
$yx$  & $+$  & $-$ & $-$ \\
$x\backslash y$  & $-$  & $+$ & $-$ \\
$y\backslash x$  & $-$  & $-$ & $-$ \\
$y\slash x$  & $-$  & $-$ & $-$ \\
$x\slash y$  & $-$  & $-$ & $-$ \\
\hline
\end{tabular}
\end{tabular}\\

\begin{tabular}{c c c }
\begin{tabular}{|c| |c |c |c| }
\hline
$(P_a, R^{-1}_a, \varepsilon)$  & $f$  & $e$ & $s$ \\
\hline
$xy$  & $-$  & $-$ & $-$ \\
$yx$  & $-$  & $-$ & $-$ \\
$x\backslash y$  & $-$  & $+$ & $-$ \\
$y\backslash x$  & $-$  & $-$ & $-$ \\
$y\slash x$  & $+$  & $-$ & $-$ \\
$x\slash y$  & $-$  & $-$ & $-$ \\
\hline
\end{tabular}
&
\begin{tabular}{|c| |c |c |c| }
\hline
$(P_a,  P_a, \varepsilon)$  & $f$  & $e$ & $s$ \\
\hline
$xy$  & $-$  & $-$ & $-$ \\
$yx$  & $-$  & $-$ & $-$ \\
$x\backslash y$  & $-$  & $+$ & $-$ \\
$y\backslash x$  & $+$  & $-$ & $-$ \\
$y\slash x$  & $-$  & $-$ & $-$ \\
$x\slash y$  & $-$  & $-$ & $-$ \\
\hline
\end{tabular}
&
\begin{tabular}{|c| |c |c |c| }
\hline
$(P_a, P^{-1}_a, \varepsilon)$  & $f$  & $e$ & $s$ \\
\hline
$xy$  & $-$  & $-$ & $-$ \\
$yx$  & $-$  & $-$ & $-$ \\
$x\backslash y$  & $-$  & $+$ & $-$ \\
$y\backslash x$  & $-$  & $-$ & $-$ \\
$y\slash x$  & $-$  & $-$ & $-$ \\
$x\slash y$  & $+$  & $-$ & $-$ \\
\hline
\end{tabular}
\end{tabular}\\

\begin{tabular}{c c c }
\begin{tabular}{|c| |c |c |c| }
\hline
$(P_a, \varepsilon,  L_a)$  & $f$  & $e$ & $s$ \\
\hline
$xy$  & $-$  & $-$ & $+$ \\
$yx$  & $-$  & $-$ & $-$ \\
$x\backslash y$  & $+$  & $-$ & $-$ \\
$y\backslash x$  & $-$  & $-$ & $-$ \\
$y\slash x$  & $-$  & $-$ & $-$ \\
$x\slash y$  & $-$  & $-$ & $-$ \\
\hline
\end{tabular}
&
\begin{tabular}{|c| |c |c |c| }
\hline
$(P_a, \varepsilon,   L^{-1}_a)$  & $f$  & $e$ & $s$ \\
\hline
$xy$  & $+$  & $-$ & $+$ \\
$yx$  & $-$  & $-$ & $-$ \\
$x\backslash y$  & $-$  & $-$ & $-$ \\
$y\backslash x$  & $-$  & $-$ & $-$ \\
$y\slash x$  & $-$  & $-$ & $-$ \\
$x\slash y$  & $-$  & $-$ & $-$ \\
\hline
\end{tabular}
&
\begin{tabular}{|c| |c |c |c| }
\hline
$(P_a, \varepsilon, R_a)$  & $f$  & $e$ & $s$ \\
\hline
$xy$  & $-$  & $-$ & $+$ \\
$yx$  & $-$  & $-$ & $-$ \\
$x\backslash y$  & $-$  & $-$ & $-$ \\
$y\backslash x$  & $-$  & $-$ & $-$ \\
$y\slash x$  & $+$  & $-$ & $-$ \\
$x\slash y$  & $-$  & $-$ & $-$ \\
\hline
\end{tabular}
\end{tabular}\\

\begin{tabular}{c c c }
\begin{tabular}{|c| |c |c |c| }
\hline
$(P_a, \varepsilon,  R^{-1}_a)$  & $f$  & $e$ & $s$ \\
\hline
$xy$  & $-$  & $-$ & $+$ \\
$yx$  & $+$  & $-$ & $-$ \\
$x\backslash y$  & $-$  & $-$ & $-$ \\
$y\backslash x$  & $-$  & $-$ & $-$ \\
$y\slash x$  & $-$  & $-$ & $-$ \\
$x\slash y$  & $-$  & $-$ & $-$ \\
\hline
\end{tabular}
&
\begin{tabular}{|c| |c |c |c| }
\hline
$(P_a, \varepsilon,   P_a)$  & $f$  & $e$ & $s$ \\
\hline
$xy$  & $-$  & $-$ & $+$ \\
$yx$  & $-$  & $-$ & $-$ \\
$x\backslash y$  & $-$  & $-$ & $-$ \\
$y\backslash x$  & $-$  & $-$ & $-$ \\
$y\slash x$  & $-$  & $-$ & $-$ \\
$x\slash y$  & $+$  & $-$ & $-$ \\
\hline
\end{tabular}
&
\begin{tabular}{|c| |c |c |c| }
\hline
$(P_a, \varepsilon,   P^{-1}_a)$  & $f$  & $e$ & $s$ \\
\hline
$xy$  & $-$  & $-$ & $+$ \\
$yx$  & $-$  & $-$ & $-$ \\
$x\backslash y$  & $-$  & $-$ & $-$ \\
$y\backslash x$  & $+$  & $-$ & $-$ \\
$y\slash x$  & $-$  & $-$ & $-$ \\
$x\slash y$  & $-$  & $-$ & $-$ \\
\hline
\end{tabular}
\end{tabular}

\end{center}
}
\end{table}

\newpage

\begin{table}[hptb]
\label{TABLE_thre}
\large{
\begin{center}
\begin{tabular}{c c c }
\begin{tabular}{|c| |c |c |c| }
\hline
$(\varepsilon, P_a, L_a)$  & $f$  & $e$ & $s$ \\
\hline
$xy$  & $-$  & $-$ & $-$ \\
$yx$  & $-$  & $-$ & $+$ \\
$x\backslash y$  & $-$  & $-$ & $-$ \\
$y\backslash x$  & $-$  & $+$ & $-$ \\
$y\slash x$  & $-$  & $-$ & $-$ \\
$x\slash y$  & $-$  & $-$ & $-$ \\
\hline
\end{tabular}
&
\begin{tabular}{|c| |c |c |c| }
\hline
$(\varepsilon, P_a, L^{-1}_a)$  & $f$  & $e$ & $s$ \\
\hline
$xy$  & $-$  & $-$ & $-$ \\
$yx$  & $-$  & $+$ & $+$ \\
$x\backslash y$  & $-$  & $-$ & $-$ \\
$y\backslash x$  & $-$  & $-$ & $-$ \\
$y\slash x$  & $-$  & $-$ & $-$ \\
$x\slash y$  & $-$  & $-$ & $-$ \\
\hline
\end{tabular}
&
\begin{tabular}{|c| |c |c |c| }
\hline
$(\varepsilon, P_a, R_a)$  & $f$  & $e$ & $s$ \\
\hline
$xy$  & $-$  & $-$ & $-$ \\
$yx$  & $-$  & $-$ & $+$ \\
$x\backslash y$  & $-$  & $-$ & $-$ \\
$y\backslash x$  & $-$  & $-$ & $-$ \\
$y\slash x$  & $-$  & $-$ & $-$ \\
$x\slash y$  & $-$  & $+$ & $-$ \\
\hline
\end{tabular}
\end{tabular}\\

\begin{tabular}{c c c }
\begin{tabular}{|c| |c |c |c| }
\hline
$(\varepsilon, P_a, R^{-1}_a)$  & $f$  & $e$ & $s$ \\
\hline
$xy$  & $-$  & $+$ & $-$ \\
$yx$  & $-$  & $-$ & $+$ \\
$x\backslash y$  & $-$  & $-$ & $-$ \\
$y\backslash x$  & $-$  & $-$ & $-$ \\
$y\slash x$  & $-$  & $-$ & $-$ \\
$x\slash y$  & $-$  & $-$ & $-$ \\
\hline
\end{tabular}
&
\begin{tabular}{|c| |c |c |c| }
\hline
$(\varepsilon, P_a, P_a)$  & $f$  & $e$ & $s$ \\
\hline
$xy$  & $-$  & $-$ & $-$ \\
$yx$  & $-$  & $-$ & $+$ \\
$x\backslash y$  & $-$  & $-$ & $-$ \\
$y\backslash x$  & $-$  & $-$ & $-$ \\
$y\slash x$  & $-$  & $+$ & $-$ \\
$x\slash y$  & $-$  & $-$ & $-$ \\
\hline
\end{tabular}
&
\begin{tabular}{|c| |c |c |c| }
\hline
$(\varepsilon, P_a, P^{-1}_a)$  & $f$  & $e$ & $s$ \\
\hline
$xy$  & $-$  & $-$ & $-$ \\
$yx$  & $-$  & $-$ & $+$ \\
$x\backslash y$  & $-$  & $+$ & $-$ \\
$y\backslash x$  & $-$  & $-$ & $-$ \\
$y\slash x$  & $-$  & $-$ & $-$ \\
$x\slash y$  & $-$  & $-$ & $-$ \\
\hline
\end{tabular}
\end{tabular}\\

\begin{tabular}{c c c }
\begin{tabular}{|c| |c |c |c| }
\hline
$(P^{-1}_a, L_a, \varepsilon)$  & $f$  & $e$ & $s$ \\
\hline
$xy$  & $+$  & $-$ & $-$ \\
$yx$  & $-$  & $-$ & $-$ \\
$x\backslash y$  & $-$  & $-$ & $-$ \\
$y\backslash x$  & $-$  & $-$ & $-$ \\
$y\slash x$  & $-$  & $+$ & $-$ \\
$x\slash y$  & $-$  & $-$ & $-$ \\
\hline
\end{tabular}
&
\begin{tabular}{|c| |c |c |c| }
\hline
$(P^{-1}_a,  L^{-1}_a, \varepsilon)$  & $f$  & $e$ & $s$ \\
\hline
$xy$  & $-$  & $-$ & $-$ \\
$yx$  & $-$  & $-$ & $-$ \\
$x\backslash y$  & $+$  & $-$ & $-$ \\
$y\backslash x$  & $-$  & $-$ & $-$ \\
$y\slash x$  & $-$  & $+$ & $-$ \\
$x\slash y$  & $-$  & $-$ & $-$ \\
\hline
\end{tabular}
&
\begin{tabular}{|c| |c |c |c| }
\hline
$(P^{-1}_a, R_a,  \varepsilon)$  & $f$  & $e$ & $s$ \\
\hline
$xy$  & $-$  & $-$ & $-$ \\
$yx$  & $+$  & $-$ & $-$ \\
$x\backslash y$  & $-$  & $-$ & $-$ \\
$y\backslash x$  & $-$  & $-$ & $-$ \\
$y\slash x$  & $-$  & $+$ & $-$ \\
$x\slash y$  & $-$  & $-$ & $-$ \\
\hline
\end{tabular}
\end{tabular}\\

\begin{tabular}{c c c }
\begin{tabular}{|c| |c |c |c| }
\hline
$(P^{-1}_a, R^{-1}_a, \varepsilon)$  & $f$  & $e$ & $s$ \\
\hline
$xy$  & $-$  & $-$ & $-$ \\
$yx$  & $-$  & $-$ & $-$ \\
$x\backslash y$  & $-$  & $-$ & $-$ \\
$y\backslash x$  & $-$  & $-$ & $-$ \\
$y\slash x$  & $+$  & $+$ & $-$ \\
$x\slash y$  & $-$  & $-$ & $-$ \\
\hline
\end{tabular}
&
\begin{tabular}{|c| |c |c |c| }
\hline
$(P^{-1}_a,  P_a, \varepsilon)$  & $f$  & $e$ & $s$ \\
\hline
$xy$  & $-$  & $-$ & $-$ \\
$yx$  & $-$  & $-$ & $-$ \\
$x\backslash y$  & $-$  & $-$ & $-$ \\
$y\backslash x$  & $+$  & $-$ & $-$ \\
$y\slash x$  & $-$  & $+$ & $-$ \\
$x\slash y$  & $-$  & $-$ & $-$ \\
\hline
\end{tabular}
&
\begin{tabular}{|c| |c |c |c| }
\hline
$(P^{-1}_a,  P^{-1}_a, \varepsilon)$  & $f$  & $e$ & $s$ \\
\hline
$xy$  & $-$  & $-$ & $-$ \\
$yx$  & $-$  & $-$ & $-$ \\
$x\backslash y$  & $-$  & $-$ & $-$ \\
$y\backslash x$  & $-$  & $-$ & $-$ \\
$y\slash x$  & $-$  & $+$ & $-$ \\
$x\slash y$  & $+$  & $-$ & $-$ \\
\hline
\end{tabular}
\end{tabular}

\end{center}
}
\end{table}

\newpage

\begin{table}[hptb]
\label{TABLE_thre}
\large{
\begin{center}
\begin{tabular}{c c c }
\begin{tabular}{|c| |c |c |c| }
\hline
$(P^{-1}_a, \varepsilon,  L_a)$  & $f$  & $e$ & $s$ \\
\hline
$xy$  & $-$  & $-$ & $-$ \\
$yx$  & $-$  & $-$ & $+$ \\
$x\backslash y$  & $+$  & $-$ & $-$ \\
$y\backslash x$  & $-$  & $-$ & $-$ \\
$y\slash x$  & $-$  & $-$ & $-$ \\
$x\slash y$  & $-$  & $-$ & $-$ \\
\hline
\end{tabular}
&
\begin{tabular}{|c| |c |c |c| }
\hline
$(P^{-1}_a, \varepsilon, L^{-1}_a)$  & $f$  & $e$ & $s$ \\
\hline
$xy$  & $+$  & $-$ & $-$ \\
$yx$  & $-$  & $-$ & $+$ \\
$x\backslash y$  & $-$  & $-$ & $-$ \\
$y\backslash x$  & $-$  & $-$ & $-$ \\
$y\slash x$  & $-$  & $-$ & $-$ \\
$x\slash y$  & $-$  & $-$ & $-$ \\
\hline
\end{tabular}
&
\begin{tabular}{|c| |c |c |c| }
\hline
$(P^{-1}_a, \varepsilon, R_a)$  & $f$  & $e$ & $s$ \\
\hline
$xy$  & $-$  & $-$ & $-$ \\
$yx$  & $-$  & $-$ & $+$ \\
$x\backslash y$  & $-$  & $-$ & $-$ \\
$y\backslash x$  & $-$  & $-$ & $-$ \\
$y\slash x$  & $+$  & $-$ & $-$ \\
$x\slash y$  & $-$  & $-$ & $-$ \\
\hline
\end{tabular}
\end{tabular}\\

\begin{tabular}{c c c }
\begin{tabular}{|c| |c |c |c| }
\hline
$(P^{-1}_a, \varepsilon, R^{-1}_a)$  & $f$  & $e$ & $s$ \\
\hline
$xy$  & $-$  & $-$ & $-$ \\
$yx$  & $+$  & $-$ & $+$ \\
$x\backslash y$  & $-$  & $-$ & $-$ \\
$y\backslash x$  & $-$  & $-$ & $-$ \\
$y\slash x$  & $-$  & $-$ & $-$ \\
$x\slash y$  & $-$  & $-$ & $-$ \\
\hline
\end{tabular}

\begin{tabular}{|c| |c |c |c| }
\hline
$(P^{-1}_a, \varepsilon,  P_a)$  & $f$  & $e$ & $s$ \\
\hline
$xy$  & $-$  & $-$ & $-$ \\
$yx$  & $-$  & $-$ & $+$ \\
$x\backslash y$  & $-$  & $-$ & $-$ \\
$y\backslash x$  & $-$  & $-$ & $-$ \\
$y\slash x$  & $-$  & $-$ & $-$ \\
$x\slash y$  & $+$  & $-$ & $-$ \\
\hline
\end{tabular}
&
\begin{tabular}{|c| |c |c |c| }
\hline
$(P^{-1}_a, \varepsilon, P^{-1}_a)$  & $f$  & $e$ & $s$ \\
\hline
$xy$  & $-$  & $-$ & $-$ \\
$yx$  & $-$  & $-$ & $+$ \\
$x\backslash y$  & $-$  & $-$ & $-$ \\
$y\backslash x$  & $+$  & $-$ & $-$ \\
$y\slash x$  & $-$  & $-$ & $-$ \\
$x\slash y$  & $-$  & $-$ & $-$ \\
\hline
\end{tabular}
\end{tabular}\\

\begin{tabular}{c c c }
\begin{tabular}{|c| |c |c |c| }
\hline
$( \varepsilon, P^{-1}_a, L_a)$  & $f$  & $e$ & $s$ \\
\hline
$xy$  & $-$  & $-$ & $+$ \\
$yx$  & $-$  & $-$ & $-$ \\
$x\backslash y$  & $-$  & $-$ & $-$ \\
$y\backslash x$  & $-$  & $+$ & $-$ \\
$y\slash x$  & $-$  & $-$ & $-$ \\
$x\slash y$  & $-$  & $-$ & $-$ \\
\hline
\end{tabular}
&
\begin{tabular}{|c| |c |c |c| }
\hline
$( \varepsilon, P^{-1}_a,  L^{-1}_a)$  & $f$  & $e$ & $s$ \\
\hline
$xy$  & $-$  & $-$ & $+$ \\
$yx$  & $-$  & $+$ & $-$ \\
$x\backslash y$  & $-$  & $-$ & $-$ \\
$y\backslash x$  & $-$  & $-$ & $-$ \\
$y\slash x$  & $-$  & $-$ & $-$ \\
$x\slash y$  & $-$  & $-$ & $-$ \\
\hline
\end{tabular}
&
\begin{tabular}{|c| |c |c |c| }
\hline
$(\varepsilon, P^{-1}_a, R_a)$  & $f$  & $e$ & $s$ \\
\hline
$xy$  & $-$  & $-$ & $+$ \\
$yx$  & $-$  & $-$ & $-$ \\
$x\backslash y$  & $-$  & $-$ & $-$ \\
$y\backslash x$  & $-$  & $-$ & $-$ \\
$y\slash x$  & $-$  & $-$ & $-$ \\
$x\slash y$  & $-$  & $+$ & $-$ \\
\hline
\end{tabular}
\end{tabular}\\

\begin{tabular}{c c c }
\begin{tabular}{|c| |c |c |c| }
\hline
$(\varepsilon, P^{-1}_a, R^{-1}_a)$  & $f$  & $e$ & $s$ \\
\hline
$xy$  & $-$  & $+$ & $+$ \\
$yx$  & $-$  & $-$ & $-$ \\
$x\backslash y$  & $-$  & $-$ & $-$ \\
$y\backslash x$  & $-$  & $-$ & $-$ \\
$y\slash x$  & $-$  & $-$ & $-$ \\
$x\slash y$  & $-$  & $-$ & $-$ \\
\hline
\end{tabular}
&
\begin{tabular}{|c| |c |c |c| }
\hline
$( \varepsilon, P^{-1}_a,  P_a)$  & $f$  & $e$ & $s$ \\
\hline
$xy$  & $-$  & $-$ & $+$ \\
$yx$  & $-$  & $-$ & $-$ \\
$x\backslash y$  & $-$  & $-$ & $-$ \\
$y\backslash x$  & $-$  & $-$ & $-$ \\
$y\slash x$  & $-$  & $+$ & $-$ \\
$x\slash y$  & $-$  & $-$ & $-$ \\
\hline
\end{tabular}
&
\begin{tabular}{|c| |c |c |c| }
\hline
$( \varepsilon, P^{-1}_a,  P^{-1}_a)$  & $f$  & $e$ & $s$ \\
\hline
$xy$  & $-$  & $-$ & $+$ \\
$yx$  & $-$  & $-$ & $-$ \\
$x\backslash y$  & $-$  & $+$ & $-$ \\
$y\backslash x$  & $-$  & $-$ & $-$ \\
$y\slash x$  & $-$  & $-$ & $-$ \\
$x\slash y$  & $-$  & $-$ & $-$ \\
\hline
\end{tabular}
\end{tabular}
\end{center}
}
\end{table}

{\bf Acknowledgement. } The authors thank  Professor Alexandar  Krapez for fruitful discussions.

\vspace{2mm}
\begin{parbox}{118mm}{\footnotesize  Grigorii Horosh$^{1}$, Nadeghda  Malyutina$^{2}$,  Alexandra Scerbacova$^{3}$,   Victor Shcherbacov$^{4}$
\vspace{3mm}

\noindent

$^{1}$Ph.D. Student //
Institute of Mathematics and Computer Science of Moldova
\noindent Email: grigorii.horos@math.md

\vspace{3mm}

\noindent

$^{2}$Senior Lecturer//Shevchenko Transnistria State University
\noindent
Email: 231003.bab.nadezhda@mail.ru

\vspace{3mm}

$^{3}$Ph.D. Student//
Skolkovo Institute of Science and Technology
\noindent Email: scerbik33@yandex.ru

\vspace{3mm}

\noindent
$^{4}$Principal Researcher//
Institute of Mathematics and Computer Science  of Moldovav
\noindent Email: victor.scerbacov@math.md
}
\end{parbox}

\end{document}